\documentclass[12pt]{article}
\usepackage{amscd,amsmath,amsthm,amsfonts,amssymb}
\usepackage[all]{xypic}
\usepackage{hyperref}

\newtheorem{theorem}{Theorem}[section]
\newtheorem{corollary}[theorem]{Corollary}
\newtheorem{lemma}[theorem]{Lemma}
\newtheorem{proposition}[theorem]{Proposition}

\theoremstyle{definition}
\newtheorem{definition}{Definition}[section]

\theoremstyle{remark}
\newtheorem{remark}{Remark}[section]
\newtheorem{example}{Example}[section]

\newcommand{\m}{\mathfrak{M}}
\newcommand{\di}{\mathrm{d}}
\newcommand{\dddet}{\underline{\underline{\det}}}
\newcommand{\ddet}{\underline{\det}}
\newcommand{\san}{\mathcal{S}_A^n}
\newcommand{\End}{\mathrm{End}}

\newcommand{\Hilb}{\mathrm{Hilb}}
\newcommand{\GL}{\mathrm{GL}}
\newcommand{\TS}{\mathrm{TS}}
\newcommand{\MP}{\mathrm{MP}}
\newcommand{\PP}{\mathrm{P}}
\newcommand{\NN}{\mathbb{N}}
\newcommand{\Z}{\mathbb Z}

\newcommand{\al}{\alpha}
\newcommand{\bt}{\beta}

\newcommand{\ups}{\upsilon}

\newcommand{\gm}{\gamma}
\newcommand{\A}{\mathcal{A}}

\newcommand{\C}{\mathcal{C}}

\newcommand{\nn}{\mathcal{N}}

\newcommand{\pd}[2]{#1^{(#2)}}
\newcommand{\K}{k}
\newcommand{\CC}{\mathbb{C}}

\newcommand{\ZZ}{\mathbb{Z}}

\newcommand{\ran}{\mathcal{R}ep_A^n}
\newcommand{\rana}{\ran \times _R \mathbb{A}^n_R}

\newcommand{\uan}{\mathrm{U}_A^n}
\newcommand{\van}{V_n(A)}

\title{Moduli of linear representations, symmetric products and the non commutative Hilbert scheme}
\author{Francesco Vaccarino}
\date{}

\begin{document}
\maketitle \markboth{} {}

\centerline{"Geometric methods in representation theory"}
\centerline{ Grenoble, June 16 - July 4, 2008}

\section{Introduction}
Let $k$ be a commutative ring and let $R$ be a commutative $k-$algebra.
Given a positive integer $n$  and a $R-$algebra $A$ one can consider three functors of points from the category $\C_R$ of commutative $R-$algebras to the small category of sets.
All these functors are representable, namely
\begin{itemize}
  \item $\ran$ represents the functor induced  by $B\to\hom_R(A,M_n(B))$, where $M_n(B)$ are the $n\times n$ matrices over $B$, for all $B\in\C_R$.
  \item the non-commutative Hilbert scheme $Hilb^n_A$ represents the functor induced  by
  \[B\to\{\text{ left ideals } I\subset A\otimes_kB\,:\, A\otimes_kB/I \text{ is a projective $R-$module of rank $n$}\}\]
  for all $B\in \C_R$.
  \item Spec $\Gamma ^n_R(A)^{ab} \,$ represents the functor induced by
  \[
    B\to\{\text{ multiplicative polynomial laws homogeneous of degree $n$ from $A$ to $B$ }\}
 \]
 for all $B\in \C_R$.
\end{itemize}

When $A$ is commutative $Hilb^A_n$ is the usual Hilbert scheme of $n-$points of $X=\mathrm{Spec}\,A$.
A polynomial law is a kind of map generalizing polynomial mapping and coinciding with it over flat $R-$modules. The typical example of multiplicative polynomial law homogeneous of degree $n$ is the determinant of $n\times n$ matrices. The $R-$algebra $\Gamma ^n_R(A)^{ab}$ is the quotient of the $R-$algebra $\Gamma ^n_R(A)$ of the divided powers of degree $n$ on $A$ by the ideal generated by commutators . When $A$ is flat as $R-$module this coincides with the symmetric tensors of order $n$ that is $\Gamma ^n_R(A)\cong (A^{\otimes n})^{S_n}$, where $S_n$ is the symmetric group. Therefore when $A$ is commutative and flat we have $\mathrm{Spec}\,\Gamma ^n_R(A)^{ab}\cong X^{(n)}$, the $n-$th symmetric product of $X=\mathrm{Spec}\,A$.

We discuss the connections between the coarse moduli space $\ran//GL_n$ of the $n-$dimensional representations of $A,\,$ with $Hilb_A^n$ and the affine scheme Spec $\Gamma ^n_R(A)^{ab} \,$. We build a norm map from $Hilb_A^n$ to $\Gamma^n_R(A)^{ab} \,$ which specializes to the Hilbert-Chow morphism on the geometric points when $A$ is commutative and $k$ is an algebraically closed field. This generalizes the construction done by Grothendieck, Deligne and others. When $k$ is an infinite field
and $A=k\{x_1,\dots,x_m\}$ is the free $k-$associative algebra on $m$ letters, we use the isomorphism Spec $\Gamma ^n_k(A)^{ab}\cong
\ran//GL_n$ following from Theorem 1 in \cite{V1} to give a simple description of this norm map.

A field of application of our construction can be the extension to the positive characteristic case of some of the results due to L. Le Bruyn on noncommutative desingularization which can be found in \cite{lebook}. Another intriguing possibility is to use our construction to extend the work done by C.H.Liu and S.T.Yau on D-Branes in \cite{Y1} to the more recent non commutative case in \cite{Y2}.

This survey paper is based on \cite{GV,V2,V3,V1,V5}.

\medskip
\centerline{\bf{Acknowledgements}}
I would like to thank the organizers for their invitation. The author is supported by Progetto di Ricerca Nazionale COFIN 2007
"Teoria delle rappresentazioni: aspetti algebrici e geometrici ."
\section*{Notations}\label{not} Unless otherwise stated we adopt the following
notations:
\begin{itemize}
\item $k$ is a fixed commutative ground ring.
\item $R$ is a commutative $k-$algebra.
\item $B$ is a commutative $R-$algebra.
\item $A$ is an arbitrary $R-$algebra.
\item $F=k\{x_1,\dots,x_m\}$ denotes the associative free $k-$algebra on
$m$ letters.
\item $\mathcal{N}_-,\,\mathcal{C}_-, Mod_-$ and $Sets$ denote the
categories of $-$algebras,
commutative $-$algebras, $-$modules and sets, respectively.
\item we write $\A(B,C)$ for the $\hom_{\A}(B,C)$  set in a category $\A$ with
$B,C$ objects in $\A$.
\end{itemize}

\section{Moduli of representations}\label{funct}

\subsection{The universal representation }
We denotes by $M_n(B)$ the full ring of $n \times n$ matrices over
$B.\,$ If $f \ : \ B \to C $ is a ring homomorphism we denote with
$ M_n(f)\ : \ M_n(B) \to M_n(C) $ the homomorphism induced on
matrices.
\begin{definition}
By an { \em n-dimensional representation of} $A$ over $B$ we mean
a homomorphism of $R-$algebras $ \rho\, : A \, \to M_n(B).\,
$\end{definition} The assignment $B\to \nn_R(A,M_n(B))$ defines a
covariant functor $\mathcal {N}_{R} \rightarrow Sets .\,$ This
functor is represented by a commutative $R-$algebra. We report
here the proof of this fact to show how this algebra comes up
using generic matrices. These objects will be also crucial in the
construction of the norm map in Section \ref{nmap}.

\begin{lemma}\cite[Lemma 1.2.]{DPRR}\label{repr}
For all $A\in\nn_R$ there exist a commutative $R-$algebra $V_n(A)$
and a representation $\pi_A:A\to M_n(V_n(A))$ such that
$\rho\mapsto M_n(\rho)\cdot\pi_A$ gives an isomorphism
\begin{equation}\label{proc}
\C_R(V_n(A),B)\xrightarrow{\cong}\nn_R(A,M_n(B))\end{equation} for
all $B\in C_R$.
\end{lemma}
\begin{proof}
Let $A$ be an arbitrary $R-$algebra and suppose that there exist $\pi_A$ and $V_n(A)$ as in the statement. Let $I\subset A$ be a bilateral ideal
of $A$. Let $J$ be the ideal generated by $\pi_A(I)$ in $M_n(V_n(A))$. There exists then an ideal $K\subset V_n(A)$ such that $J=M_n(K)$.
It should be clear that $V_n(A/I)=V_n(A)/K$ and that $\pi_{A/I}:A/I\to M_n(V_n(A/I))$ is given by
\[
\xymatrix{A\ar[r]^{\pi _A} \ar[d] & M_n(V_n(A))\ar[d]\\
A/I\ar[r] _(.2){\pi _{A/I}} & M_n(V_n(A/I))=M_n(V_n(A))/J}
\]

It is enough now to prove the statement for $A=R\{x_1,\dots,x_m\}=R\otimes_k F$  the free
associative $R-$algebra on $m$ letters. Let $V_n(A)=R[\xi_{kij}]$
be the polynomial ring in variables $\{\xi_{kij}\,:\,
i,j=1,\dots,n,\, k=1,\dots,m\}$ over the base ring $R$. To every
$n-$dimensional representation of $A$ over $B$ it corresponds a
unique $m-$tuple of $n\times n$ matrices, namely the images of
$x_1,\dots,x_m$, hence a unique
$\bar{\rho}\in\C_R(R[\xi_{kij}],B)$ such that
$\bar{\rho}(\xi_{kij})=(\rho(x_k))_{ij}$. Following C.Procesi
{\cite{Pr1,DPRR}} we introduce the generic matrices. Let
$\xi_k=(\xi_{kij})$ be the $n\times n$ matrix whose $(i,j)$ entry
is $\xi_{kij}$ for $i,j=1,\dots,n$ and $k=1,\dots,m$.  We call
$\xi_{1},\dots,\xi_{m}$ the generic $n\times n$ matrices. Consider
the map
\[
\pi_A:A \to M_n(V_n(R)),\;\;\;\;x_{k}\longmapsto \xi_{k},\;\;\;
k=1,\dots,m\,.
\]
It is then clear that the map $\C_R(V_n(A),B)\ni \sigma\mapsto
M_n(\sigma)\cdot\pi_A\in\nn_R(A,M_n(B))$ gives the isomorphism
(\ref{proc}) in this case.
\end{proof}
\begin{remark}
It should be clear that the number of generators $m$ of $A$ is
immaterial and we can extend the above isomorphism to an arbitrary $R-$algebra.
\end{remark}
\begin{definition}\label{pigreco}\label{class}
We write $\ran$ to denote Spec\,$V_n(A).\,$ It is considered as an
$R-$scheme. \\ The map
\begin{equation}
\pi _A :A \to M_n(V_n(A)),\;\;\;\;x_{k}\longmapsto \xi_{k}.
\end{equation}
is called {\em the universal n-dimensional representation.}

\smallskip
Given a representation $\rho : A \to M_n(B)$ we denote by
$\bar{\rho}$ its classifying map $\bar{\rho}:V_n(A)\to B\,$.
\end{definition}

\begin{example}\label{fcase}
For the free algebra one has $\ran\cong M_n^m$ the scheme whose
$B-$points are the $m-$tuples of $n\times n$ matrices with entries in $B$.
\end{example}

\begin{example}\label{commuting}
Note that $\ran$ could be quite complicated, as an example, when
$A=\CC [x,y]$ we obtain that $\ran$ is the \textit{commuting
scheme} i.e. the couples of commuting matrices and it is not even
known (but conjecturally true) if it is reduced or not, see
\cite{V3}.
\end{example}

\subsection{$\GL_n-$action}\label{repaction}
\begin{definition}
We denote by $\GL_{n}$ the affine group scheme whose group of $B-$points form the group $ \GL_n(B)$ of $n \times n$ invertible
matrices with entries in $B$, for all $B\in\C_R$.
\end{definition}
Define a $\GL_{n}$-action on $\ran$ as follows. For any $\varphi
\in \ran (B),\, g \in \GL_{n}(B),\,$ let $\varphi ^g : \van \to B$
be the $R-$algebra homomorphism corresponding to the
representation given by
\begin{equation}
\begin{matrix}
A & \to & M_n(B) \\
a & \longmapsto & g (M_n(\varphi)\cdot \pi _A(a))g^{-1}.
\end{matrix}
\end{equation}
Note that if $\varphi, \varphi '$ are $B-$points of $\ran ,\,$then
the $A-$module structures induced on $B^n$ by $\varphi$ and
$\varphi '$ are isomorphic if and only if there exists $g \in
\GL_{n}(B)$ such that $\varphi ' = \varphi ^g .$
\begin{definition}\label{modrep}
We denote by $\ran//\GL_n=\mathrm{Spec}\,V_n(A)^{\GL_{n}(R)}\,$
the categorical quotient (in the category of $R-$schemes) of
$\ran$ by $\GL_{n}.\,$ It is the \textit{(coarse) moduli space of
$n-$dimensional linear representations} of $A$.
\end{definition}

\subsubsection{Determinant}
\begin{definition}
We denote by $C_{n}(A)$ the $k-$subalgebra of $V_{n}(A)$ generated by the coefficients $e_i(\pi_A(f))$ of characteristic polynomials
\[\det(t-\pi_A(f))=t^n+\sum_{i=1}^n(-1)^i e_i(\pi_A(f))t^{n-i}\]
as $f$ varies in $A$.
\end{definition}
\begin{remark}
Let $B\in\C_R$ and let $b\in M_{n}(B)$ and $1\leq i \leq n$. Note that $e_i(b)$ is  the trace of
$\wedge^{i}(b)$. Obviously $e_1(b)$ is the trace of $b$ and $e_n(b)=\det(b)$ is the determinant of $b$.
\end{remark}

\begin{remark}\label{isocla}
We have just seen that the general linear group $\GL_n(R)$ acts on $V_{n}(A)$. It is clear that $C_n(A)\subset V_n(A)^{\GL_n(R)}$ for all $k,\, R\in\C_k$ and $A\in\nn_k$. It has been shown in {\cite{Do,Pr3,Zu}} that
$C_{n}(F)=V_{n}(F)^{\GL_n(k)}$ when $k$ is $\Z$ or an infinite field. When $k$ is a characteristic zero field it follows that $C_n(A)= V_n(A)^{\GL_n(k)}$ for all $A\in\nn_k$ because $\GL_n(k)$ is linear reductive in this case.
\end{remark}

\section{Symmetric products and divided powers}\label{poly}
In this part of the paper we introduce a representable functor whose representing scheme will play the same role played by the symmetric product in the classical Hilbert-Chow morphism.

\subsection{Polynomial laws}

We first recall the definition of polynomial laws between
$k$-modules. These are mappings that generalize the usual
polynomial mappings between free $k$-modules. We mostly follow {\cite{Ro1,Ro2, V5, zipgen} and we refer the interested reader to
these papers for detailed descriptions and proof.

\begin{definition} Let $M$ and $N$ be two
$k$-modules. A \emph{polynomial law} $\varphi$ from $M$ to $N$ is
a family of mappings $\varphi_{_{A}}:A\otimes_{k} M
\to A\otimes_{k} N$, with $A\in\C_{k}$ such that the
following diagram commutes
\begin{equation}
\xymatrix{
A\otimes_{k}M \ar[d]_{f\otimes id_M} \ar[r]^{\varphi_A}
& A\otimes_{k} N \ar[d]^{f\otimes id_N} \\
B\otimes_{k} M \ar[r]_{\varphi_B}
& B\otimes_{k} N }
\end{equation}
for all $A,\,B\in\C _k$ and all $f\in \C _k(A,B)$.
\end{definition}
\begin{definition}
Let $n\in \NN$. If $\varphi_A(au)=a^n\varphi_A(u)$, for all $a\in
A$, $u\in A\otimes_{k} M$ and all $A\in\C_k$, then $\varphi$
is called \emph{homogeneous of degree} $n$.
\end{definition}
\begin{definition}
If $M$ and $N$ are two $k$-algebras and
\[
\begin{cases} \varphi_A(xy)&=\varphi_A(x)\varphi_A(y)\\
\varphi_A(1_{A\otimes M})&=1_{A\otimes N}
\end{cases}
\]
for $A\in\C_{k}$ and for all $x,y\in A\otimes_{k} M$, then
$\varphi$ is called \emph{multiplicative}.
\end{definition}
We need the following
\begin{definition}\label{finsu}
For $S$ a set and any additive monoid $M$, we denote by $M^{(S)}$
the set of functions $f:S\rightarrow M$ with finite support.

Let $\al\in M^{(S)}$, we denote by $\mid \al \mid$ the (finite)
sum $\sum_{s\in S} \al(s)$,
\end{definition}

Let $A$ and $B$ be two $k-$modules and $\varphi:A\rightarrow B$ be
a polynomial law. The following result on polynomial laws is a
restatement of Th\'eor\`eme I.1 of {\cite{Ro1}}.
\begin{theorem}\label{roby} Let $S$ be a set.
\begin{enumerate}
\item Let $L=k[x_s]_{s\in S}$ and let $a_{s}$ be elements of $A$ such
that $a_{s}$ is $0$ except for a finite number of $s\in S$, then
there exist $\varphi_{\xi}((a_{s}))\in B$, with $\xi \in \NN^{(S)}$,
such that:
\begin{equation}\varphi_{_{L}}(\sum_{s\in S} x_s\otimes
a_{s})=\sum_{\xi \in \NN^{(S)}} x^{\xi}\otimes
\varphi_{\xi}((a_{s}))\end{equation}
where $x^{\xi}=\prod_{s\in S}
x_s^{\xi_s}$.
\item Let $R$ be any commutative $k-$algebra and let
$r_s\in R$ for $s\in S$, then:
\begin{equation}\varphi_{_{R}}(\sum_{s\in S}
r_s\otimes a_{s})=\sum_{\xi \in \NN^{(S)}} r^{\xi}\otimes
\varphi_{\xi}((a_{s}))\end{equation}
where $r^{\xi}=\prod_{s\in S}
r_s^{\xi_s}$.
\item If $\varphi$ is homogeneous of degree $n$, then one has
$\varphi_{\xi}((a_{s}))=0$ if $| \xi |$ is
different from $n$. That is:
\begin{equation}\varphi_{_{R}}(\sum_{s\in S}
r_s\otimes a_s)=\sum_{\xi \in \NN^{(S)},\,| \xi |=n}
r^{\xi}\otimes \varphi_{\xi}((a_s))\end{equation}
In particular, if $\varphi$ is
homogeneous of degree $0$ or $1$, then it is constant or linear,
respectively.
\end{enumerate}
\end{theorem}
\begin{remark}\label{coeff}
The above theorem means that a polynomial law $\varphi:A\rightarrow
B$ is completely determined by its coefficients
$\varphi_{\xi}((a_{s}))$, where $\{a_s\,:\,s\in S\}$ is a set of linear generator of $A$ and $\xi \in \NN^{(S)}$.
\end{remark}
\begin{remark}\label{freelaw} If $A$ is a free $k-$module and $\{a_{t}\,
:\, t\in T\}$ is a basis of $A$, then $\varphi$ is completely
determined by its coefficients $\varphi_{\xi}((a_{t}))$, with $\xi
\in \NN^{(T)}$. If also $B$ is a free $k-$module with basis
$\{b_{u}\, :\, u\in U\}$, then $\varphi_{\xi}((a_{t}))=\sum_{u\in
U}\lambda_{u}(\xi)b_{u}$. Let $a=\sum_{t\in T}\mu_{t}a_{t}\in A$.
Since only a finite number of $\mu_{t}$ and $\lambda_{u}(\xi)$ are
different from zero, the following makes sense:
\begin{eqnarray*}\varphi(a)=\varphi(\sum_{t\in T}\mu_{t}a_{t}) =
\sum_{\xi\in
\NN^{(T)}} \mu^{\xi}\varphi_{\xi}((a_{t}))& = & \sum_{\xi\in
\NN^{(T)}} \mu^{\xi}(\sum_{u\in U}\lambda_{u}(\xi)b_{u})\\ & = &
\sum_{u\in U}(\sum_{\xi\in \NN^{(T)}}\lambda_{u}(\xi)
\mu^{\xi})b_{u}.\end{eqnarray*} Hence, if both $A$ and $B$ are free
$k-$modules, a polynomial law $\varphi:A\rightarrow B$ is simply a
polynomial map.
\end{remark}
\begin{definition}\label{polset}
Let $k$ be a commutative ring.
\begin{itemize}
\item[(1)] For $M,N$ two $k-$modules we set $\PP_k^n(M,N)$ for
the set of homogeneous polynomial laws $M\rightarrow N$ of degree
$n$.
\item[(2)] If $M,N$ are two $k-$algebras we set $\MP_k^n(M,N)$ for
the multiplicative homogeneous polynomial laws $M\rightarrow N$ of
degree $n$.
\end{itemize}
\end{definition}
The assignment $N\to \PP_k^n(M,N)$ (resp.
$N\to \MP^n_k(M,N)$) determines a functor from
$Mod _k$ (resp. $\mathcal {N}_k$) to $Sets$.
\begin{example}\label{expol}
\begin{enumerate}
\item For all $A,B\in\nn_k$ it holds that $\MP_k^1(A,B)=\nn_k(A,B)$. A linear
multiplicative polynomial law is a $k-$algebra homomorphism.
\item For $B\in\C_k$ the usual determinant $\mathop{det}:M_n(B)\to B$
belongs to $\MP_k^n(M_n(B),B)$
\item For $B\in\C_k$ the mapping $b\mapsto b^n$ belongs to $\MP_k^n(B,B)$
\item For $B,C\in\C_B$ consider $p\in\MP_B^n(B,C)$, since
$p(b)=p(b\,1)=b^np(1)=b^n\,1_C$ it is clear that in this case raising to
the power $n$ is the unique multiplicative polynomial law homogeneous of
degree $n$.
\item When $A\in\nn_k$ is an Azumaya algebra of rank $n^2$ over its
center $k$, then its reduced norm $N$ belongs to $\MP_k^n(A,k)$.
\end{enumerate}
\end{example}
\subsection{Divided powers}
The functors just introduced in Def. \ref{polset} are represented
by the divided powers which we introduce right now.
\begin{definition}\label{divdef}
For a $k$-module
$M$ the {\em divided powers algebra} $\Gamma_{k}(M)$ (see
\cite{Ro1,Ro2}) is an
associative and commutative $k$-algebra with identity $1_{k}$
and product $\times$, with generators $\pd m k $, with $m\in M$,
$k \in \Z$ and relations, for all $m,n\in M$:
\begin{enumerate}
\item $\pd m i = 0, \forall i<0$;
\item $\pd m 0 = 1_{\scriptstyle{k}}, \forall m\in M$;
\item $\pd {(\al m)} i = \al^i\pd m i, \forall \al\in k,
\forall i\in \NN$;
\item $\pd {(m+n)} k = \sum_{i+j=k}\pd m i \times \pd n j , \forall
k\in \NN$;
\item\label{shuf} $\pd m i \times \pd m j = {i+j\choose i}\pd m {i+j} ,
\forall i,j\in \NN$.
\end{enumerate}
\end{definition}
The $k$-module $\Gamma_{k}(M)$ is generated by finite
products $\times_{i\in I}\, \pd {x_i} {\al_i}$ of the above
generators. The divided powers algebra $\Gamma_{k}(M)$ is an
$\NN$-graded algebra with homogeneous components
$\Gamma^n_{k}(M)$, ($n\in \NN$), the submodule generated by
$\{\times_{i\in I} \pd {x_i} {\al_i}\;
:\;\mid\al\mid=\sum_i\al_i=n\}$. One easily checks that
$\Gamma_{k}$ is a functor from $Mod _k$ to $\mathcal C _k$.
\begin{remark} Suppose $n!$ is invertible in $k$ e.g. $k\supset \mathbb{Q}$. The map given by $\pd x n \mapsto x^{n}/n!$ induces an isomorphism of $k-$algebras between $\Gamma_{k}(M)$ and the symmetric algebra over $M$. Whence the name {\it{divided powers}}.

\end{remark}

\subsection{Universal properties}
\subsubsection{Functoriality and adjointness}\label{upgamma}
$\Gamma_{k}^n$ is a covariant functor from $Mod_{k}$ to $Mod_{k}$
and one can easily check that it preserves surjections. The map
$\gamma^n:r\mapsto \pd r n$ is a polynomial law
$M\rightarrow\Gamma^n_{k}(M)$ homogeneous of degree $n$. We call
it \emph{the universal map} for the following reason: consider
another $k-$module $N$ and the set $Mod_{k}(\Gamma^n_{k}(M),N)$ of
homomorphisms of $k-$modules between $\Gamma^n_{k}(M)$ and $N$, we
have an isomorphism
\begin{equation}Mod_{k}(\Gamma^n_{k}(M),N)\xrightarrow{\cong}
\PP^n_{k}(M,N)\end{equation}
given by
$\varphi\mapsto\varphi\circ\gamma^n$.
\subsubsection{The algebra $\Gamma_{k}^n(A)$}
If $A$ is $k-$algebra then $\Gamma^n_{k}(A)$ inherits a
structure of $k-$algebra by
\begin{equation}\pd a n \star\pd b n=\pd{(ab)}n\end{equation}
The unit in $\Gamma_{k}^n(A)$ is $\pd 1 n$. It was proved by
N.Roby {\cite{Ro2}} that in this way $A\to\Gamma_{k}^n(A)$ gives a
functor from $k-$algebras to $k-$algebras such that
$\gamma^n(a)\gamma^n(b)=\gamma^n(ab)$, $\forall a,b \in A$. Hence
\begin{equation}\label{mp}
\nn_{k}(\Gamma^n_{k}(A),B)\xrightarrow{\cong} \MP^n_{k}(A,B)\end{equation}
and the map given by $\varphi\mapsto\varphi\circ\gamma^n$ is an isomorphism
for all $A,B\in
\nn_{k}$.
\subsubsection{Base change}\label{bc}\cite[Thm. III.3, p. 262]{Ro1}
For any $R\in \C_k$ and $A\in Mod_k$ it holds that
\begin{equation}R\otimes_k\Gamma_k(A)\xrightarrow{\cong}\Gamma_R(R\otimes_k
A).\end{equation}
When $A$ is a $k-$algebra this gives an isomorphism of $R-$algebras
\begin{equation}\label{bcn}
R\otimes_k\Gamma_k^n(A)\xrightarrow{\cong}\Gamma_R^n(R\otimes_k
A)\end{equation}
for all $n\geq 0$.

\subsection{Symmetric tensors}\label{gt}
\begin{definition}\label{symt}
Let $M$ be a $k-$module and consider the $n-$fold tensor power $M^{\otimes n}$.
The symmetric group $S_n$ acts on
$M^{\otimes n}$ by permuting the factors and we denote by $\TS^n_{k}(M)$
or simply by $\TS^n(M)$ the
$k-$submodule of $M^{\otimes n}$ of the invariants for this action. The
elements of $\TS^n(M)$ are called
symmetric tensors of degree $n$ over $M$.
\end{definition}
\begin{remark}
If $M$ is a $k-$algebra then $S_n$ acts on $M^{\otimes n}$ as a group of
$k-$algebra automorphisms. Hence
$\TS^n(M)$ is a $k-$subalgebra of $M^{\otimes n}$.
\end{remark}
\subsubsection{Flatness and Symmetric Tensors}\label{flat}
Suppose $M\in\nn_{k}$ (resp. $M\in\C_{k}$). The homogeneous polynomial law
$M\rightarrow \TS_k^n(M)$ given by $x\mapsto x^{\otimes n}$ gives
a morphism $\tau_n:\Gamma_{k}^n(M)\rightarrow \TS_k^n(M)$ which is
an isomorphism when $M$ is flat over $k$. Indeed one can easily prove
that $\tau_n$ is an isomorphism in case $M$ is free. The flat case then follows because any flat $k-$module is a filtered direct limit of free modules and both $\Gamma^n$ and $TS^n$ commute with filtered direct limits.
\begin{remark}
Divided powers and symmetric tensors are not always isomorphic. See
{\cite{Lu}} for counterexamples.
\end{remark}

\subsection{Generators}\label{gen}

Following \cite{V1} we give here a generating set for $\Gamma_k^n(A)$.

We need a Lemma.
\begin{lemma}\label{basis}
Suppose $M\in Mod_k$ is generated by $\{m_i\}_{i\in I}$ then $\Gamma^n_{k}(M)$ is generated by
$\{\times_{i\in I} \pd {m_i} {\al_i}\;
:\;\mid\al\mid=\sum_i\al_i=n\}$ for all $n\in \NN$.
In particular if $\{m_i\}_{i\in I}$ is a basis of $M$ then $\{\times_{i\in I} \pd {m_i} {\al_i}\;
:\;\mid\al\mid=\sum_i\al_i=n\}$ is a basis of $\Gamma_k^n(M)$ for all $n\in \NN$.
\end{lemma}
\begin{proof}
The first assertion follows directly from the definition.

Suppose $M$ is free of finite rank $p$. By Remark \ref{freelaw} it follows that any polynomial law $M\to k$ homogeneous of degree $n$ can be identified with an homogeneous polynomial in $p$ variables $\sum_{h_1+\dots+h_p=n} a_{h_1,\dots,h_p}x_1^{h_1}\cdots x_p^{h_p}\in k[x_1,\dots,x_p]$. This means that one can find a linear form taking arbitrary values on the $\pd {m_1} {\al_1}\times \cdots \times \pd {m_p} {\al_p}$. This argument easily extends to not finitely generated free $k-$modules by applying it to arbitrary finite subsets of a basis.
\end{proof}

\begin{definition}\label{degree}
Let $\m$ denote the set of monomials in $F$. There is a natural
degree $\di$ on $F$ given by $\di(x_i)=1$ for all $i=1,\dots ,m$
and $\di(1)=0$\,. We denote by $\m^+$ the set of monomials of
positive degree. Thus $\m=\m^+\cup\{1\}$\,. We set $F^+$ for the
ideal of $F$ generated by $\m^+$.
\end{definition}

It is clear that $\m$ is a basis of $F$ so that $\m_n=\{\times_{\mu \in \m} \pd {\mu} {\al_{\mu}}\;
:\;\mid\al\mid=n\}$ is a basis of $\Gamma_k^n(M)$ for all $n\in \NN$ by Lemma \ref{basis}.

\begin{proposition}[Product Formula]\label{pdf}
Let $h,k\in\NN$, $\al\in\NN^h$, $\bt\in\NN^k$ be such that $|\al |,\,|\bt| \leq n$. Let
$r_1,\dots,r_h,s_1,\dots,s_k\in F^+$. Then
\begin{multline}\pd 1 {n-\mid\al\mid}\times_{i=1}^h\pd {r_i} {\al_i} \star \pd 1 {n-\mid\bt\mid}\times_{j=1}^k \pd {s_j} {\bt_j}
=\\
=\sum_{\gm} \pd 1 {n-\mid\gm\mid} \times_{i=1}^h \pd {r_i} {\gm_{i0}} \times_{j=1}^k \pd {s_j} {\gm_{0j}}
\overset{h,k}{\underset{i,j=1}{\times}} \pd {(r_i s_j)} {\gm_{ij}}
\end{multline}
where
\[\gamma=(\gamma_{10},\gamma_{20},\dots,\gamma_{h0},\gamma_{01},\ldots,\gamma_{0k},
\gamma_{11},\ldots,\gamma_{1k},\ldots,\gamma_{h1},\ldots,\gamma_{hk})\]
is such that
\begin{equation}\label{sys}
\begin{cases} \gamma_{ij}\in \NN \\
\sum_{i,\,j} \gamma_{ij}\leq n \\
\sum_{j=0}^k \gamma_{ij}=\al_i \;\;{\text{for}}\;\; i=1,\dots,h\\
\sum_{i=0}^h \gamma_{ij}=\bt_j \;\; {\text{for}}\;\; j=1,\dots,k.
\end{cases}
\end{equation}
\end{proposition}
\begin{proof}
Let $t_1,t_2$ be two commuting independent variables and let $a,b\in F^+$. We have
\begin{equation}\label{aperb}
\pd {(1+t_1\otimes a)} n \star \pd {(1+t_2\otimes b)} n=\pd {(1+t_1\otimes a+t_2\otimes b+t_1t_2\otimes ab)} n\end{equation} hence
\begin{multline*}
(\pd 1 n + \sum_{i=1}^n t_1^i\otimes \pd 1 {n-i} \times \pd a i)
\star(\pd 1 n + \sum_{j=1}^n t_2^j \otimes \pd 1 {n-j} \times \pd b j)=\\
=\pd 1 n + \sum_{i,j}t_1^it_2^j\otimes \pd 1 {n-i} \times \pd a i\star \pd 1 {n-j} \times \pd b j=\\
=\pd 1 n +\sum _{l_1,\,l_2,\,l_{12}}t_1^{l_1+l_{12}}t_2^{l_2+l_{12}}\otimes
\pd 1 {n-\mid\gm\mid} \times_{i=1}^h \pd {r_i} {\gm_{i0}} \times_{j=1}^k \pd {s_j} {\gm_{0j}}
\overset{h,k}{\underset{i,j=1}{\times}} \pd {(r_i s_j)} {\gm_{ij}}
\end{multline*}
The desired equation then easily follows.
\end{proof}
\begin{remark}
The product formula could be more easily visualized by observing that we are summing over those matrices of positive integers
\[\overline{\gamma }=
\left(
  \begin{array}{cccc}
    0& \gamma_{01} &\ldots & \gamma_{0k}\\
    \gamma_{10} &\gamma_{11} & \ldots& \gamma_{1k}\\
    \vdots &\vdots & \ddots& \vdots\\
    \gamma_{h0} & \gamma_{h1} &\ldots& \gamma_{hk} \\
  \end{array}
\right)
\]
having the last $h$ rows and the last $k$ columns with sums $\al_1,\dots,\al_h$ and $\bt_1,\dots,\bt_k$ respectively.
\end{remark}
\begin{remark}
The above Product Formula can be derived from the one found by N.Roby in the context of divided powers (see {\cite{Ro1}} ). It has also been  derived by D.Ziplies in his paper on the divided powers algebra $\widehat{\Gamma}$ (see\cite{zipgen} ).
\end{remark}
\begin{corollary}\label{subring} Let $k\in \NN$, $a_1,\dots ,a_k\in F^+$,
$\al=(\al_1,\dots,\al_k)\in \NN^k$ with $|\al|\leq n$. Then $\pd 1 {n-\mid\al\mid}\times_{i=1}^k \pd {a_i}{\al_i}$ belongs to
the subalgebra of\, $\Gamma_k^n(F)$ generated by the $\pd 1 {n-i}\times\pd {\mu} i$, where $i=1,\dots,n$ and $\mu$ is a monomial in
the $a_1,\dots,a_k$.
\end{corollary}
\begin{proof}
We prove the claim by induction on $|\al|$ assuming that $\al_i>0$ for all $i$ (note that $1 \leq k\leq
\sum_j\al_j$). Since $n$ is fixed we suppress the superscript $n$ for all the proof.

If $\sum_j\al_j=1$ then $k=1$ and $\pd 1 {n-\mid\al\mid}\times_{i=1}^k \pd {a_i}{\al_i}=\pd 1 {n-1}\times \pd {a_1} 1$. Suppose the claim true for
all $\pd 1 {n-\mid\bt\mid}\times_{i=1}^h \pd {b_i}{\bt_i}$ with $b_1,\dots,b_h\in F^+$ and $|\bt| < |\al|$.

Let $k, a_1,\dots ,a_k, \al$ be as in the statement, then we have by the Product Formula
\begin{multline}
\pd 1 {n-1}\times \pd {a_1} {\al_1}\star \pd 1 {n-(\al_2+\dots+\al_k)}\times_{i=2}^k \pd {a_i}{\al_i} = \pd 1 {n-\mid\al\mid}\times_{i=1}^k \pd {a_i}{\al_i}+\\
\sum_{\gm} \pd 1 {n-\mid\gm\mid} \pd {a_1} {\gm_{10}} \times_{j=2}^k \pd {a_j} {\gm_{0j}} \times_{l=2}^k \pd {(a_1 a_l)} {\gm_{1l}}
\end{multline}
where
\[\gamma=(\gamma_{10},\gamma_{02},\dots,\gamma_{0k},\gamma_{12},\dots,\gamma_{1k})\]
with $\gamma_{10}+\gamma_{12}+\dots +\gamma_{1k}=\al_1$ with $\sum_{l=2}^k \gamma_{1l}>0$, and
$\gamma_{0j}+\gamma_{1j}=\al_j$ for $j=2,\dots,k$.
Thus
\[\gamma_{10}+\gamma_{02}+\dots+\gamma_{0k}+\gamma_{12}+\dots+\gamma_{1k}
=\sum_j\al_j - \sum_{l=2}^k \gamma_{1j}<\sum_j\al_j.\] Hence
\begin{multline} \pd 1 {n-\mid\al\mid}\times_{i=1}^k \pd {a_i}{\al_i}=\pd 1 {n-1}\times \pd {a_1} {\al_1}\star \pd 1 {n-(\al_2+\dots+\al_k)}\times_{i=2}^k \pd {a_i}{\al_i}\\ - \sum_{\gm} \pd 1 {n-\mid\gm\mid} \pd {a_1} {\gm_{10}} \times_{j=2}^k \pd {a_j} {\gm_{0j}} \times_{l=2}^k \pd {(a_1 a_l)} {\gm_{1l}}
\end{multline}
where $|\gm|=\sum_{r,s}\gm_{rs}<|\al|$. So the
claim follows by the induction hypothesis.
\end{proof}
\begin{corollary}\label{gen}
The algebra $\Gamma_k^n(F)$ is generated by the $\pd 1 {n-i}\times \pd {\mu} i$ where $1\leq i\leq n$
and $\mu\in\m^+$.
\end{corollary}
\begin{proof}
It follows from Corollary \ref{subring} applied to the elements of the basis $\m_n$.
\end{proof}
\begin{remark}
The above corollaries can also be proved using Corollary (4.1) and (4.5) in \cite{zipgen}.
\end{remark}
\begin{lemma}\label{ple} For all $f\in F^+$, and $k,h\in \NN$,
$\pd 1 {n-h} \times \pd {(f^k)} h$ belongs to the subalgebra of $\Gamma_k^n(F)$ generated by the $\pd 1 {n-j} \times \pd f j$.
\end{lemma}
\begin{proof}
Let $f\in F^+$. The evaluation $k[x]\rightarrow F$ induces an homomorphism
$\rho_f:\Gamma_k^n(k[x])\to \Gamma_k^n(F)$
Now it is clear that
\[\Gamma_k^n(k[x])\cong \TS_k^n(k[x]) \cong k[x_1,\dots,x_n]^{S_n}\]
so that
\[(1+t\otimes x)^{(n)}\mapsto (1+t\otimes x)^{\otimes n} \mapsto \prod_{i=1}^n(1+t\otimes x_i)=1+\sum_{i=1}^n t^i\otimes e_i(x_1,\dots,x_n)\]
where the $e_i(x_1,\dots,x_n)$ are the elementary symmetirc functions and $t$ is and independent variable.
We have another distinguished kind of functions in the ring of symmetric polynomials $k[x_1,\dots,x_n]^{S_n}$ beside the elementary symmetric ones: the
\emph{power sums}. For $r\geq 1$ the $r$-th power sum is
\begin{equation}\label{pow}
p_r=\sum_{i\geq 1}x_i^r\end{equation}
Let $g\in k[x_1,\dots,x_n]^{S_n}$, set $g\cdot p_r=g(x_1^r,x_2^r,\dots,x_n^r)$ for the plethysm of $g$ and $p_r$ (see Section
I.8 of \cite{M}). The function $g \cdot p_r$ is again symmetric. Since the $e_i$ freely generate $k[x_1,\dots,x_n]^{S_n}$ we
have that $g\cdot p_r$ can be expressed as a polynomial in the $e_i$ and we denote it by
\begin{equation} P_{h,k}=e_h\cdot p_k\label{plet}\end{equation}
It follows that
\begin{multline}\pd 1 {n-h}\times \pd {(f^k)} h=\rho_f(e_h\cdot p_k)=\rho_f(P_{h,k}(e_1,e_2\dots,e_n))=\\
=P_{h,k}(\pd 1 {n-1}\times \pd f 1,\pd 1 {n-2}\times\pd f 2,\dots,\pd f n)\end{multline}
and the result is proved.
\end{proof}
\begin{definition}
A monomial $\mu\in\m^+$ is called \emph{primitive} if it is not the proper power of another one.
\end{definition}
\begin{example}$x_1x_2x_1x_2$ is not primitive while $x_1x_2x_1x_1$ is
primitive.
\end{example}
We have then the following refinement of Corollary \ref{gen}.
\begin{theorem}[Generators]\label{prigen}
The algebra $\Gamma_k^n(F)$ is generated by $\pd 1 {n-i}\times\pd {\mu} i$ with $1\leq i \leq n$ and $\mu$ primitive.
\end{theorem}
\begin{proof}
It follows from Corollary~\ref{gen} and Lemma~\ref{ple}.
\end{proof}

\subsection{The abelianization} In this paragraph we introduce the
abelianization functor and we prove some of its properties. Then
we deduce the fact that the abelianization of divided powers
commutes with base change.
\begin{definition}
Given a $k-$algebra $A$ we denote by $[A]$ the two-sided ideal of
$A$ generated by the commutators $[a,b]=ab-ba$ with $a,b\in A$. We
write
\[A^{ab}=A/[A]\]
and call it the abelianization of $A$.

We denote by $\mathfrak{ab}_A$ the surjective homomorphism $\mathfrak{ab}_A:A\to A^{ab}$\,.

For $A,B\in \nn_k$ and $f\in\nn_k(A,B)$ we denote by $f^{ab}\in \C_k(A^{ab},B^{ab})$ the abelianization of $f$.
\end{definition}
We collect some facts regarding this construction.
\begin{proposition}\label{abefun}
\begin{enumerate}
\item For all $B\in \C_k$ there is an isomorphism $\C_k(A^{ab},B)\to \nn_k(A,B)$ by means of $\rho\mapsto\rho\cdot\mathfrak{ab}_A$. Equivalently for all $\varphi\in\nn_k(A,B)$ there is a unique $\overline{\varphi}:A^{ab}\to B$ such that the following diagram commutes
    \[
\xymatrix{
  A \ar[dr]_{\varphi} \ar[r]^{\mathfrak{ab}_A}
                & A^{ab} \ar[d]^{\overline{\varphi}}  \\
                & B             }
\]
\item The assignment $A\to A^{ab}$ induces a covariant functor $\nn_k\to\C_k$ that preserves surjections.
    \end{enumerate}
\end{proposition}
\begin{proof}
The first point is obvious.
For $A,B\in \nn_k$ and $f\in\nn_k(A,B)$ we have that $f([A])\subset [B]$. This proves the functoriality.
Let $A,B\in \nn_k$ and suppose $f\in\nn_k(A,B)$ to be surjective. We have that $[B]=f([A])$.
\end{proof}
\begin{definition}
We call the just introduced functor the {\textit{abelianization functor}}.
\end{definition}
\begin{proposition}\label{abcom}
For all $A\in\nn_k$ and all $R\in\C_k$ it holds that
\[(R\otimes_k A)^{ab}\cong R\otimes_k A^{ab}\]
\end{proposition}
\begin{proof}
The homomorphism $id_R\otimes \mathfrak{ab}_A:R\otimes_k A \to R\otimes_k A^{ab}$ of $R-$algebras induces a unique one $\alpha:(R\otimes_kA)^{ab}\to R\otimes_k A^{ab}$ making the following diagram commutative
  \[
\xymatrix{
  R\otimes_k A \ar[dr]_{\mathfrak{ab}_{R\otimes A}} \ar[r]^{id_R\otimes\mathfrak{ab}_A}
                & R\otimes_k A^{ab}  \\
                & (R\otimes_k A)^{ab} \ar[u]^{\alpha}             }
\]
On the other hand the homomorphism of $k-$algebras given by the composition \[A\to R\otimes_k A\xrightarrow{\mathfrak{ab}_{R\otimes A}}(R\otimes_k A)^{ab}\] induces a unique one $\beta:A^{ab}\to (R\otimes_k A)^{ab}$. This $\beta$ extends to a homomorphism of $R-$algebras $\tilde{\beta}:R\otimes_k A^{ab}\to  (R\otimes_k A)^{ab}$ that is the inverse of $\alpha$ as can be easily checked.
\end{proof}

\begin{corollary}\label{bcab} With the notation of Proposition \ref{abcom} we have
\[\Gamma_R^n(R\otimes_k A)^{ab}\cong R\otimes_k\Gamma_k^n(A)^{ab}\]
\end{corollary}
\begin{proof}
It follows from (\ref{bcn}) and Proposition~\ref{abcom}.
\end{proof}
\begin{remark}
Corollary \ref{bcab} remain true by replacing $\Gamma^n$ with $\TS^n$ when $A$ is flat.
In particular when $R$ is a field.
\end{remark}
\begin{definition}\label{san}
We denote by $\san$ the affine $R-$scheme Spec $\Gamma_R^n(A)^{ab}$.
\end{definition}
\begin{proposition}\label{sanbed}
Suppose $A=R \otimes F/J \in \nn_r$ then the induced morphism
$\san\to \mathcal{S}_{R\otimes F}^n \cong
\mathcal{S}_F^n\times_k\mathrm{Spec}\,R$ is a closed immersion.
\end{proposition}
\begin{proof}
It follows from Proposition \ref{abefun} and \ref{bc}.
\end{proof}
\begin{remark}
In view of Definition \ref{san} all the results of this section can be rephrased in terms of $R-$schemes.
In particular Corollary \ref{bcab} means that $\san$ base-changes well.
\end{remark}
\begin{remark}\label{xn}
When $A$ is flat (see paragraph \ref{flat}) and commutative one sees that
\[\san\cong X^n/S_n=\,\mathrm{Spec}\,\TS^n_R(A)\]
the $n-$fold symmetric product of $X=$Spec $A$.

In particular it follows that $\mathcal{S}_F^n\cong
\mathrm{Spec}\,\TS_k^n(F)^{ab}$.
\end{remark}

\subsection{Generators II}\label{gen2}
We extend here the results of Section \ref{gen} to the abelianization $\Gamma(A)^{ab}$.
\begin{definition}
Consider $\m^+/\sim$ the set of the equivalence classes of monomials $\mu\in\m^+$ where $\mu\sim\mu'$ if and
only if there is a cyclic permutation $\sigma$ such that $\sigma(\mu)=\mu'$. We set $\Psi$ to denote the set of
equivalence classes in $\m^+/\sim$ made of  primitive monomials
\end{definition}

\begin{theorem}\label{genab}
The ring $\Gamma^n_k(F)^{ab}$ is generated as a ring by $\mathfrak{ab}(\pd 1 {n-i}\times \pd {\mu} i)$ where $1\leq i \leq n$ and $\mu$ varying in a complete set of representatives of $\Psi$.
\end{theorem}
\begin{proof}
Using (\ref{aperb}) it easy to see that
\begin{equation}\mathfrak{ab}(\pd 1 {n-i}\times\pd {(rs)} i)=\mathfrak{ab}(\pd 1 {n-i}\times\pd {(sr)} i)\end{equation}
for all $1\leq i \leq n$ and $r,s\in F^+$. The result then follows from Theorem \ref{prigen} and the surjectivity
of $\mathfrak{ab}$.
\end{proof}

\begin{corollary}\label{sip}
Let $A$ be a $k-$algebra generated by $\{a_i\}_{i\in I}$ then
\begin{enumerate}
\item $\Gamma_{k}^n(A)$ is generated by $\pd 1 {n-j}\pd {\mu} j$ where
$1\leq j\leq n$ and $\mu$ varies in the set of primitive monomials
in $\{a_i\}_{i\in I}$
\item $(\Gamma_{k}^n(A))^{ab}$ is
generated by $\mathfrak{ab}(\pd 1 {n-j}\pd {\ups} j)$ where $\ups$ varies in a
complete set of representative of equivalence classes (under cyclic
permutations) of primitive monomials in the $\{a_i\}_{i\in I}$.
\end{enumerate}
\end{corollary}
\begin{proof}
Given the $k-$algebra $A$ and a set $\{a_i\}_{i\in I}$ of its
generators we have a surjective homomorphism
$F_I=\K\{x_i\}_{i\in I} \to A$ hence another surjective one
$\Gamma_k^n(F_I)\to\Gamma_{k}^n(A)$
and the result follows from Theorem \ref{prigen}.
The result on the abelianization follows from Theorem \ref{genab} since
the abelianization functor preserves surjections.
\end{proof}
\begin{remark}
Corollary \ref{sip}(1) is a refinement of Cor.(4.5) in
\cite{zipgen}.
\end{remark}
\begin{corollary}\label{rh}

\begin{enumerate} The following hold
\item
The ring
$\TS_k^n(\K[x_1,\dots,x_m])\cong\K[x_{11},x_{12},\dots,x_{m1},\dots,x_{mn}]^{S_n}$
of the multisymmetric functions also known as the ring of the vector
invariants of the symmetric group is generated by the $\pd 1 {n-i}\times \pd {\ups} i$
with $\ups=x_1^{\al_1}\cdots x_{m}^{\al_m}$ such that
$\al_1,\dots,\al_m$ are coprime.
\item Let $A$ be a commutative $k-$algebra generated by
$\{a_i\}_{i\in I}$ then $\Gamma_{k}^n(A)$ is generated by $\pd 1
{n-j}\pd {\ups} j$ where $\ups=\prod_i a_i^{\al_i}$ is such that
$\sum_i\al_i$ is finite and the $\al_i$ are coprime.
\end{enumerate}
\end{corollary}
\begin{proof}
It follows from Theorem \ref{genab} since $\Gamma^n$ preserves
surjections. Write $P=K[x_1,\dots,x_m]$, then $P$ is a free
$k-$module hence the surjection $F\to P$ induces
another surjection $TS_k^n(F)^{ab}\to TS_k^n(P)$. Use
Theorem \ref{genab}.
\end{proof}
\begin{remark}
This gives another proof of Th.1 in \cite{V4}
\end{remark}
\begin{remark} A {\em minimal\/} generating set for $\Gamma_k^n(k[x_1,\dots,x_m])$ has been found by David Rydh \cite{rd} improving Corollary \ref{rh}.
\end{remark}

\section{The norm map}\label{nmap} In this section we introduce
a morphism which connect $\ran//\GL_n$ and $\san.$

Let $\rho\in\nn_R(A,B)$ be a representation of $A$ over
$B\in\C_R$. The composition $\det\cdot\rho$ is a multiplicative
polynomial law homogeneous of degree $n$ belonging to
$MP_R^n(A,B)$. We denote by $\det_{\rho}$ the unique homomorphism
in $\C_R(\Gamma_R^n(A)^{ab},B)$ such that
$\det\cdot\rho=\det_{\rho}\cdot\gamma^n$, see (\ref{mp}). The
correspondence $\rho\mapsto\det_{\rho}$ is clearly functorial in
$B$ giving then a morphism $\Gamma_R^n(A)^{ab}\to V_n(A)$ by
universality and henceforth a morphism of $R-$schemes
\begin{equation}\label{ddet}
\ddet:\ran\to\san
\end{equation}
Let $\bar{\rho}:V_n(A)\to B$ be the unique $B-$point of $\ran$ such that $\rho=M_n(\bar{\rho})\cdot\pi_A$, see Definition \ref{pigreco}.
  It is easy to check that
 \[\bar{\rho}\cdot\det\cdot\pi_A=\det\cdot M_n(\bar{\rho})\cdot\pi_A=\det\cdot\rho\]
 so that the following commutes
 \begin{equation}\label{detro}
\xymatrix{
A \ar[r]^(.3){\pi _A} \ar@{=}[d]&
M_n(V_n(A))\ar[d]^{M_n(\bar{\rho})}\ar[r]^{\det}&V_n(A)\ar[d]^{\bar{\rho}} \\
A\ar[r]^(.3){\rho}& M_n(B)\ar[r]^{\det}&B }
\end{equation}
It follows that $\ddet$ is the affine morphism corresponding to the composition of the universal map $\pi _A$ introduced in ({\ref{pigreco}}) with the determinant i.e. to the top horizontal arrows of  diagram (\ref{detro}).

The determinant is invariant under basis changes and $\ran//\GL_n$ is a categorical quotient. Hence there exists a unique morphism $\dddet:\ran//\GL_n\to\san$ such that the following commutes
\begin{equation}\label{invcnm}
\xymatrix{\ran\ar[rd]^{\ddet}\ar[d]\\
\ran//\GL_n\ar[r]_(.6){\dddet}&\san}
\end{equation}
We have the following result.
\begin{theorem}
The morphism $\dddet : \ran//\GL_n\rightarrow\san$ has the following properties.
\begin{enumerate}
\item If $A=k\{x_1,\dots,x_m\}$ then $\dddet$ is an isomorphism
when $k=R$ is an infinite field or $k=R=\ZZ$. \item Suppose $k=R$
is an infinite field. If $A\in\C_k$ then $\dddet$ induces an
isomorphism between the associated reduced schemes.
\item Suppose
$k=R$ is a characteristic zero field then
    \subitem when $A\in\nn_k$ then $\dddet$ is a closed embedding,
    \subitem when $A\in\C_k$ then $\dddet$ is an isomorphism.
\item When $A$ is commutative and flat as $R-$module then
$\Gamma_R(A)\cong\TS_R^n(A)\to V_n(A)$ is an injective
homomorphism.
\end{enumerate}
\end{theorem}
\begin{proof}
First of all observe that $\dddet$ gives $\pd 1 {n-i}\times \pd a i \mapsto e_i(\pi_A(a))$ for all $i=1,\dots, n$ and $a\in A$. Therefore by Corollary \ref{sip} we have that $\dddet$ is a surjection $\Gamma_k(A)\to C_n(A)\subset V_n(A)^{\GL_n(R)}$ by Remark \ref{isocla}.

1. When $A$ is the free $k-$algebra $\Gamma_k(A)^{ab}\cong C_n(A)=V_n(A)^{\GL_n(k)}$ by Theorem 1.1 in \cite{V1}. This is obtained by showing that $\Gamma_k(A)^{ab}$ and $C_n(A)=V_n(A)^{\GL_n(k)}$ shares a common presentation by generators and relations.

2. and 4. When $A$ is commutative and flat as $k-$module the above mapping has an inverse $C_n(A)\to \TS^n_k(A)\cong \Gamma_k^n(A)$ given by the restriction to $C_n(A)$ of the unique homomorphism $V_n(A)\to A^{\otimes n}$ that corresponds to the linear representation $A\to M_n(A^{\otimes n})$ defined by mapping $a\in A$ to the diagonal matrix whose diagonal is $(a\otimes 1^{\otimes n-1},1\otimes a \otimes 1^{n-2},\dots,1^{\otimes n-1}\otimes a)$. This gives $\Gamma_k^n(A)\cong C_n(A)$. In this way it is obvious that $\dddet$ corresponds to the restriction to diagonal matrices. It is also clear that this gives a surjection $V_n(A)^{\GL_n(k)}\to \Gamma_k^n(A)$.
When $k=R$ is an infinite field we have an isomorphism of the associated reduce schemes. Indeed the surjection $V_n(A)^{\GL_n(k)}\to \Gamma_k^n(A)$ induced by the restriction to diagonal matrices is also injective because any orbit has in its closure a tuple of diagonal matrices, i.e. a semisimple representation.

3.The fact that $\dddet$ is a closed embedding in characteristic zero follows considering the following diagram
\[
\xymatrix{
\mathrm{Rep}_F^n//\GL_n \ar[r]_(.6){\cong}^(.6){\dddet} & \mathrm{S}_F^n \\
\ran//\GL_n \ar[r]^(.6){\dddet} \ar[u] & \san \ar[u]}
\]
recalling Proposition \ref{sanbed} and the fact that $\GL_n$ in this case is linear reductive.
This fact together with the isomorphism of reduced structures give the isomorphism for $A\in\C_k$ when $k$ is a characteristic zero field.
\end{proof}

\section{The non commutative Hilbert scheme }\label{nchs}
For $A\in\nn_R$ we recall the definition and the main properties of a
representable functor of points whose representing scheme is the usual Hilbert scheme of points
when $A$ is commutative. For references see for example \cite{Ba,GV,No,Re,Se1,Se2,VdB}.

\begin{definition}
For any algebra $B\in\mathcal{C}_R$ we write
\[
\begin{array}{ll}
\Hilb_A^n(B) := & \{\mbox {left ideals } I
\mbox { in } A \otimes_{R} B \mbox { such that } M=A \otimes _R B /I \
\mbox {is projective} \\
& \mbox{ of rank } n
\mbox { as a } B\mbox{-module} \}.
\end{array}
\]
\end{definition}

\begin{proposition}
The correspondence $\mathcal{C}_R\to Sets$ induced by $B\mapsto
\Hilb_A^n(B)$ gives a covariant functor denoted by
$\Hilb_A^n$.
\end{proposition}
\begin{proof}
Straightforward verification.
\end{proof}
\begin{proposition}\cite[Proposition 2]{VdB}
The contravariant functor $R-Schemes\to Sets$ induced
by $\Hilb_A^n$ is representable by an $R-$scheme denoted by ${Hilb}^n_A\,$.
\end{proposition}
\begin{proof}
The functor $\mathrm{Hilb}^n_A$ is a closed subfunctor of the grassmanian
functor.
\end{proof}

\smallskip

Let now $B$ be a commutative $k-$algebra. Consider triples $(\rho,m,M)$
where $M$ is a projective $B-$module of rank $n,\,\rho:A \to
\End_B(M)$ is a $k-$algebra homomorphism such that $\rho(R)\subset B$ and
$\rho(A)(Bm)= M.\,$
\begin{definition}\label{triples}
The triples $(\rho,m,M)$ and $(\rho ',m',M')$ are {\em equivalent}
if there exists a $B-$module isomorphism $\alpha : M \to
M'$ such that $\alpha(m)=m'$ and
$\alpha \rho (a)\alpha ^{-1}=\rho '(a),\,$ for all $a \in A.\,$
\end{definition}

These equivalence classes represent $B-$points of $Hilb^n_A\,$ as
stated in the following
\begin{lemma}\label{bpoints}
If $B$ is a $k-$algebra, the $B-$points of $Hilb^n_A$ are in
one-one correspondence with equivalence classes of triples
$(\rho,m,M).$
\end{lemma}
\begin{proof}
Let $I\in Hilb^n_A(B).\,$ Choose $M=A \otimes _R B /I$ and $\rho :A
\rightarrow \End_B(M)$ given by the composition of the left regular
action of $A$ on itself and the $B-$module homomorphism $A \otimes _R B
\rightarrow M.\,$ Finally let $m=1_M.\,$

On the other hand, let $(\rho,m,M)$ be as in the statement, we
consider the map
\begin{equation}
A \otimes _R B\to M, \;\;\;\;\; a \otimes b \longmapsto
\rho(a)(bm).
\end{equation}
This map is surjective and its kernel $I$ is a $B-$point in
$Hilb_A^n.\,$
If we consider a triple $(\rho ',m',M')\,$ in the same equivalence class
of $(\rho,m,M),\,\,$ we have that
\[\sum \rho'(a)(b\al(m))=\sum \al\rho(a)(b\al^{-1}\al(m))=\al\sum\rho(a)(bm)\]
so that $I'=I$ and the two triples define the same $B-$point in $Hilb^n_A.\,$
(See also \cite[Lemma 0.1]{Se1}, and \cite[Lemma 3]{VdB}).
\end{proof}

\smallskip

\begin{remark}\label{ralg}
If the triple $(\rho,m,M)$ represents a $B-$point of $Hilb^n_A ,\,$
then the homomorphism $\rho$ induces an $R-$algebra structure on $B.\,$
Moreover a homomorphism $\rho '$ where $(\rho ',m',M')$ is in the same
equivalence class of $(\rho,m,M),\,$ induces the same $R-$algebra structure
on $B.\,$
\end{remark}

\subsection{Examples}

\subsubsection{Free algebras}\label{free}
Any $A\in\nn_R$ is a quotient of an opportune free
$R$-algebra $R\{x_{\al}\}$  and in
this case it is easy to see that the scheme $Hilb_A^n$ is a closed subscheme of
$Hilb_{R\{x_{\al}\}}^n .\,$

\subsubsection{Van den Bergh's results}
In \cite{VdB} $Hilb^n_A$ was also defined by M. Van den Bergh in the
framework of Brauer-Severi schemes. He proved that the scheme $Hilb_F^n$ is
irreducible and smooth of dimension $n+(m-1)n^2 \,$ if $k$ is an
algebraically closed field (see \cite[Theorem 6]{VdB}).

\subsubsection{Commutative case: Hilbert schemes of $n$-points.}
\label{commcase} Let now $R=k$ be an algebraically closed field and let  $A$ be commutative. Let $X=\mathrm{Spec}\,A$, the $k-$points of $Hilb_A^n$ parameterize zero-dimensional
subschemes $Y\subset X$ of length $n.\,$ It is the simplest case of Hilbert scheme
parameterizing closed subschemes of $X$ with fixed Hilbert polynomial $P,\,$  in this case $P$ is the constant polynomial $n.\,$ The scheme $\Hilb_A^n$ is
usually called the {\em Hilbert scheme of $n-$points on} $X \,$ (see for example \cite{BK,LST}).

\subsection{A principal bundle over the Hilbert scheme}\label{bund}
For any $B\in\C_R\,$ identify $B^n$ with $\mathbb{A}^n_R(B),\,$
the
$B-$points of the $n-$dimensional affine scheme over $R.\,$
We introduce another functor that is one of the cornerstones of our construction.

\begin{definition}
For each $B\in \C_R$, let $\uan(B)$ denote the set of $B-$points $(\bar{\rho},v)$ of $\rana$
such that $\rho(A)(Bv)=B^n$, i.e. such that $v$ generates $B^n$ as $A-$module via $\rho:A\to M_n(B)$.
\end{definition}

\begin{remark}
It is easy to check that the assignment $B\mapsto \uan(B)$ is functorial in $B$.
Therefore we get a subfunctor $\uan$ of $\rana$ that is clearly open and we denote by $\uan$ the open subscheme of $\rana$ which represents it.
\end{remark}

We can define an action of $\GL_n\,$ on
$\rana $ similar to the one on $\ran$, namely for any $B \in \mathcal {C}_R$ let
\begin{equation}\label{act1}
g(\alpha,v)=( \alpha ^g,gv),\, g \in \GL_n(B),\, \alpha \in
\ran(B), \,v \in \mathbb{A}^n_R(B).
\end{equation}
It it is clear that $\uan$ is stable under the above action. Therefore we have that $\uan$ is an open $\GL_n-$subscheme of $\rana$.

\begin{proposition}\label{bala}
For the action described above, $\uan\to\uan/\GL_n$ is a locally-trivial principal $\GL_n-$bundle which is a universal categorical quotient.
\end{proposition}
\begin{proof}
In \cite[Proposition 1]{No} the proposition is proved for the case $A=F$ and $k=R=\ZZ$. It has been extended in \cite[Theorem 7.16]{Ba} for arbitrary $R$. The statement then follows by observing that $\rana$ is a closed $\GL_n-$subscheme of
$\mathrm{Rep}_F^n \times _R \mathbb{A}^n_R$.
\end{proof}

We have the following result.
\begin{theorem}
The scheme $\uan/\GL_n$ represents $Hilb_A^n$ and $\uan\to Hilb_A^n$ is a universal categorical quotient and a $\GL_n-$principal bundle.
\end{theorem}
\begin{proof}
In \cite[Proposition 1]{No} the proposition is proved for the case $A=F$ and $k=R=\ZZ$. It has been extended in \cite[Theorem 7.16]{Ba} for arbitrary $R$. The statement then follows by observing that $\rana$ is a closed $\GL_n-$subscheme of
$\mathrm{Rep}_F^n \times _R \mathbb{A}^n_R$.
\end{proof}

\begin{remark}
Madhav Nori gave a direct proof of the representability of $Hilb_A^n$ for $k=R=\ZZ$
and $A=F\,$ in \cite{No}. In this case Lemma \ref{bpoints} says that $B-$points in $ Hilb_F^n$ are represented by equivalence classes of triples $(\rho,v,k^n)$ where
\[
\rho : F \to M_n(k)
\]
is an $n-$dimensional representation of the algebra $F$ over $k$ and $v\in
k^n$ is such that $\rho(F)v=k^n.\,$
It is then easy to see that $\uan\to\uan/\GL_n$ is a principal $\GL_n-$bundle and that the above equivalence classes are in bijection with the $\GL_n-$orbits of $\uan$ (see the proof of Lemma \ref{bpoints}). The isomorphism $\uan/\GL_n\cong Hilb_A^n$ is then obtained by proving that $Hilb_A^n$ is $\mathrm{Proj}\,P$ where $P=\oplus _d P_d$ with
\[P_d=\{f\, :\, f(g(a_1,\dots,a_m,v))=(\det g)^d f((a_1,\dots,a_m,v)),\, \mbox{for all}\, g \in \GL_n(k)\}.
\]
\end{remark}
\begin{example}
Let $A=\CC[x,y]$ then $Hilb_A^n(\CC)$ is described by the above Theorem as
\[\{(X,Y,v)\,:\,X,\,Y\in M_n(\CC),\,XY=YX,\,\CC[X,Y]v=\CC^n\}.\]
This description of the Hilbert scheme of $n-$points of $\CC^2$ is one of the key ingredients of the celebrated Haiman's proof of the $n!$ Theorem \cite{Ha}. It has also been widely used by H.Nakajima \cite{Na}.
\end{example}

\section{Morphisms}\label{mor}
In this section we introduce morphisms which connect $\ran//\GL_n,\,
Hilb^n_A$ and $\san.$

\subsection{The forgetful map}\label{proj}
Recall from Section \ref{bund} that we have $\uan\to\uan/\GL_n\cong Hilb_A^n$. We have then a commutative diagram
\begin{equation}\label{univ}
\xymatrix{
&\rana\ar@{->>}[rd]&\\
\uan \ar@{^{(}->}[ru] \ar[rr] \ar@{->>}[d]  && \ran\ar@{->>}[d]\\
Hilb_A^n \ar[rd]\ar[rr]^{p} && \ran//\GL_n \\
& \rana//\GL_n\ar[ru]&}
\end{equation}

\begin{theorem}\label{proje}
 The morphism
\[
 p: Hilb_A^n \to \ran//\GL_n
\]
in (\ref{univ}) is projective.
\end{theorem}
\begin{proof}
By Theorem 7.16 and Remark 7.17 in \cite{Ba} this is true for $A=F$. The result follows since $Hilb_A^n$ is a closed subscheme of $Hilb_F^n$ and $\ran//\GL_n$ is affine.
\end{proof}
The fibers of the map $p$ are very difficult to study. Some results is known for the case $A=F$ and $k$ algebraically closed field \cite{Le}.

\subsection{Non commutative Hilbert-Chow}

\begin{definition}
We denote by $hc$ the morphism given by
\begin{equation}\label{CD1}
\xymatrix{
Hilb ^n _A \ar@/_1pc/[rr]_{hc} \ar[r]^(.4){p} & \ran//\GL_n \ar[r]^(.6){\dddet}&\san
 }
\end{equation}
\end{definition}

There is the following
\begin{theorem}\label{nforget}
The morphism $hc$ is projective.
\end{theorem}
\begin{proof}
We know $p$ to be projective by Theorem \ref{proje}. Since $\dddet$ is affine and hence separated the result follows.
\end{proof}

\begin{theorem}\label{final}
The image of the non commutative Hilbert-Chow morphism is isomorphic to the one of the forgetful map in the following cases
\begin{enumerate}
\item when $A$ is the free algebra and $k=R=\ZZ$ or it is an infinite field.
\item when $A$ is commutative and $k=R$ is a characteristic zero field.
\end{enumerate}
\end{theorem}
\begin{proof}
In the listed cases we have that $\dddet$ is an isomorphism hence we have an isomorphism between the image of $hc$ with the one of the forgetful map.
\end{proof}
\begin{example}
Suppose $A=F$ the free algebra. In this case $\ran//\GL_n=M_n^m(k)//\GL_n$ the coarse moduli space of $m-$tuples of $n\times n$ matrices modulo the action of $\GL_n$ by simultaneous conjugation.
The scheme $Hilb_A^n$ is in this case identified with the quotient by the above action of the open of $\rana$ of the tuples $(a_1,\dots,a_m,v)$ such that $v$ is cyclic i.e. $\{f(a_1,\dots,a_m)v\,:\, f\in A\}$ linearly generates $k^n$. In this case the Hilbert-Chow morphism can be identified with the map induced by $(a_1,\dots,a_m,v)\mapsto (a_1,\dots,a_m)$.

This picture works for $A=k[x_1,\dots,x_m]$ adding the extra condition that $a_ia_j=a_ja_i$ for all $i,\,j$.
\end{example}

\end{document}